\documentclass[a4paper,12pt]{amsart}

\usepackage{amsmath}
\usepackage{amssymb}
\usepackage{mathrsfs}
\usepackage{enumerate}
\usepackage{ifthen}
\usepackage{graphicx}
\usepackage[T1]{fontenc} 
\usepackage{color}

\nonstopmode \numberwithin{equation}{section}
\newtheorem{thm}{Theorem}[section]
\newtheorem{cor}{Corollary}[section]
\newtheorem{lem}{Lemma}[section]

\theoremstyle{definition}

\newcounter{minutes}
\divide\time by 60
\newcounter{hours}
\multiply\time by 60
\addtocounter{minutes}{-\time}

\newcounter {own}
\def\theown {\thesection       .\arabic{own}}

{\qed\bigskip}

\newcounter{alphabet}

\begin{document}

\title{Schwarzian Norm Estimate for Functions in Generalized Robertson Class}

\author{Sanjit Pal}
\address{Sanjit Pal,
Department of Mathematics, National Institute of Technology Durgapur,
Durgapur- 713209, West Bengal, India.}
\email{palsanjit6@gmail.com}


\def\thefootnote{}
\footnotetext{ {\tiny File:~\jobname.tex,
printed: \number\year-\number\month-\number\day,
          \thehours.\ifnum\theminutes<10{0}\fi\theminutes }
} \makeatletter\def\thefootnote{\@arabic\c@footnote}\makeatother

\begin{abstract}
Let $\mathcal{A}$ be the class of analytic functions $f$ in the unit disk $\mathbb{D}=\{z\in\mathbb{C}:|z|<1\}$ with the normalized conditions $f(0)=0$, $f'(0)=1$. For $-\pi/2<\alpha<\pi/2$ and $0\le \beta<1$, let $\mathcal{S}_{\alpha}(\beta)$ be the subclass of $\mathcal{A}$ consisting of functions $f$ that satisfy the relation  
$${\rm Re\,} \left\{e^{i\alpha}\left(1+\frac{zf''(z)}{f'(z)}\right)\right\}>\beta\cos{\alpha}\quad\text{for}~z\in\mathbb{D}.$$
In the present study, we will compute the sharp estimate of the pre-Schwarzian and Schwarzian norms for functions in $\mathcal{S}_{\alpha}(\beta)$.
\end{abstract}

\thanks{}

\maketitle
\pagestyle{myheadings}
\markboth{Sanjit Pal}{Schwarzian norm of generalized Robertson class}

\section{Introduction}
Let $\mathcal{H}$ be the set of all analytic functions $f$ in the unit disk $\mathbb{D}:=\{z\in\mathbb{C}:|z|<1\}$, and  $\mathcal{LU}$ denote the subclass of $\mathcal{H}$, which contains all locally univalent functions, i.e., $\mathcal{LU}=\{f\in\mathcal{H}: f'(z)\ne 0~\text{for}~ z\in\mathbb{D}\}$. The pre-Schwarzian and Schwarzian derivatives for a locally univalent function $f\in\mathcal{LU}$ are defined by
$$P_f(z):=\frac{f''(z)}{f'(z)}\quad \text{and}\quad S_f(z):=\left(\frac{f''(z)}{f'(z)}\right)^{'}-\frac{1}{2}\left(\frac{f''(z)}{f'(z)}\right)^2,$$
respectively. Also, the pre-Schwarzian and Schwarzian norms (the hyperbolic sup-norms) of $f\in\mathcal{LU}$ are defined by
$$||P_f||:=\sup\limits_{z\in\mathbb{D}}(1-|z|^2)|P_f
(z)| \quad \text{and}\quad ||S_f||:=\sup\limits_{z\in\mathbb{D}}(1-|z|^2)^2|S_f(z)|,$$
respectively. These norms have significant meanings in the theory of Teichm\"{u}ller spaces (see \cite{Lehto-1987}). It is well known that for a univalent function $f\in\mathcal{LU}$, the pre-Schwarzian and Schwarzian norms satisfies $||P_f||\leq 6$ and $||S_f||\leq 6$, respectively (see \cite{Kraus-1932,Nehari-1949}) and the estimates are sharp. On the other hand, it is also known that for a locally univalent function $f$ in $\mathcal{LU}$ with $||P_f||\leq 1$ (see \cite{Becker-1972,Becker-Pommerenke-1984}) or $||S_f||\leq 2$ (see \cite{Nehari-1949}), the function $f$ is univalent in $\mathbb{D}$. The constants $1$ and $2$ cannot be replaced by a smaller number. In 1976, Yamashita \cite{Yamashita-1976} proved that $||P_f ||<\infty$ if and only if $f$ is uniformly locally univalent in $\mathbb{D}$. Moreover, if $||P_f||<2$ then $f$ is bounded in $\mathbb{D}$ (see \cite{Kim-Sugawa-2002}). Due to computational difficulties, several researchers have studied the pre-Schwarzian norm compared to the Schwarzian norm (see, for example, \cite{Ali-Sanjit-2022,Sugawa-1998,Yamashita-1999} and references therein).\\

Let $\mathcal{A}$ denote the subclass of $\mathcal{H}$ consisting of functions $f$ normalized by $f(0)=f'(0)-1=0$. Therefore, for any function $f\in\mathcal{A}$ has the following Taylor series expansion
\begin{equation}\label{s-00001}
f(z)=z+\sum\limits_{n=2}^{\infty}a_nz^n.
\end{equation}
Let $\mathcal{S}$ be the set that contains all univalent functions $f$ in $\mathcal{A}$. A function $f\in\mathcal{A}$ is said to be starlike of order $\alpha$ ($0\le \alpha< 1$) if ${\rm Re\,} (zf'(z)/f(z)) > \alpha$ for $z\in\mathbb{D}$. A function $f\in\mathcal{A}$ is said to be convex of order $\alpha$ ($0\le \alpha< 1$) if  ${\rm Re\,} (1 + zf''(z)/f'(z))> \alpha$ for $z\in\mathbb{D}$. The sets of all starlike and convex functions of order $\alpha$ in $\mathcal{S}$ are denoted by $\mathcal{S}^*(\alpha)$ and $\mathcal{C}(\alpha)$, respectively. It is easy to show that for any function $f$ in $\mathcal{A}$ belongs to $\mathcal{C}(\alpha)$ if and only if $zf'\in\mathcal{S}^*(\alpha)$. Moreover, $\mathcal{S}^*(0)=:\mathcal{S}^*$ and $\mathcal{C}(0)=:\mathcal{C}$ are the usual classes of starlike and convex functions, respectively. More details about these classes, we refer the monographs \cite{Duren-1983,Goodman-book-1983}.\\

For $-\pi/2<\alpha<\pi/2$ and $0\le \beta<1$, a function $f\in\mathcal{A}$ is said to be $\alpha$-spirallike functions of order $\beta$ if
$${\rm Re\,} \left\{e^{i\alpha}\left(\frac{zf'(z)}{f(z)}\right)\right\}>\beta\cos{\alpha}\quad\text{for}~z\in\mathbb{D}.$$
The class of all $\alpha$-spirallike functions of order $\beta$ is denoted by $\mathcal{SP}_{\alpha}(\beta)$. In 1967, Libera \cite{Libera-1967} introduced the class $\mathcal{SP}_{\alpha}(\beta)$ and proved that functions in this
class are univalent for all $\alpha$ and $\beta$. In particular, functions in $ \mathcal{SP}_{\alpha}(0)=:\mathcal{SP}_{\alpha}$ are called $\alpha$-spirallike functions, which was introduced by \v{S}pa\v{c}ek \cite{Spacek-1933} in 1933. Note that a domain $\Omega$ containing the origin is called $\alpha$-spirallike if for each point $w_0$ in $\Omega$ the arc of the $\alpha$-spiral from origin to the point $w_0$ entirely lies in $\Omega$. A function $f\in\mathcal{A}$ is said to be an $\alpha$-spirallike if $f(\mathbb{D})$ is an $\alpha$-spirallike domain. Later on, Robertson \cite{Robertson-1969} introduced a new class of functions with the help of $\alpha$-spirallike functions and it is denoted by $\mathcal{S}_{\alpha}$ and it is defined by a function $f\in\mathcal{S}_{\alpha}$ if and only if $zf'(z)\in \mathcal{SP}_{\alpha}$. In the same paper, Robertson \cite{Robertson-1969} also proved that functions in $\mathcal{S}_{\alpha}$ need not be univalent in $\mathbb{D}$ and using Nehari's test, he showed that functions in $\mathcal{S}_{\alpha}$ is univalent if $\alpha$ satisfies the inequality $0<\cos{\alpha}\le x_0\approx  0.2034\cdots$. In 1975, Pfaltzgraff \cite{Pfaltzgraff-1975} proved that functions in $\mathcal{S}_{\alpha}$ are univalent if $0<\cos{\alpha}\le 1/2$. In 1971, Pinchuk \cite{Pinchuk-1971} generalized this class and defined a new class $\mathcal{S}_{\alpha}(\beta)$ by the condition that $f\in\mathcal{S}_{\alpha}(\beta)$ if and only if $zf'(z)\in \mathcal{SP}_{\alpha}(\beta)$ (see also Chichra \cite{Chichra-1975}). Functions in $\mathcal{S}_{\alpha}(\beta)$ are known as $\alpha$-Robertson functions of order $\beta$. A function $f\in \mathcal{A}$ is in the class $\mathcal{S}_{\alpha}(\beta)$ if and only if
$${\rm Re\,} \left\{e^{i\alpha}\left(1+\frac{zf''(z)}{f'(z)}\right)\right\}>\beta\cos{\alpha}\quad\text{for}~z\in\mathbb{D}.$$

One of most important and useful tool in geometric function theory is differential subordination method. With the help of differential subordination, many problems can be easily handled. A function $f\in\mathcal{H}$ is said to be subordinate to another function $g\in\mathcal{H}$ if there exists an analytic function $\omega:\mathbb{D}\rightarrow\mathbb{D}$ with $w(0)=0$ such that $f(z)=g(\omega(z))$ and it is simply denoted by $f(z)\prec g(z)$ or $f\prec g$. Moreover, if $g$ is univalent, then $f\prec g$ if and only if $f(0)=g(0)$ and $f(\mathbb{D})\subset g(\mathbb{D})$. In terms of subordination, the class $\mathcal{S}_{\alpha}(\beta)$ can be defined in the following form
\begin{equation}\label{t-00010}
f\in\mathcal{S}_{\alpha}(\beta)\iff 1+\frac{zf''(z)}{f'(z)}\prec \frac{1+Az}{1-z},\quad A=e^{-i\alpha}(e^{-i\alpha}-2\beta\cos{\alpha}).
\end{equation}

Pre-Schwarzian and Schwarzian derivatives are valuable tools for investigating the geometric properties of analytic mappings. The pre-Schwarzian norm, however, has garnered more attention than the Schwarzian norm due to the computational difficulties. Even though the classical work on the Schwarzian derivative in connection to geometric function theory had been done in \cite{Kuhnau-1971, Nehari-1949}, for various subclasses of locally univalent functions, more research on the Schwarzian derivative still needs to be done. Determining the Schwarzian norm for prominent subclasses of locally univalent functions is an attractive topic in Teichm\"{u}ller spaces. For the class of convex functions $\mathcal{C}$, the Schwarzian norm satisfies $||S_f||\le 2$, and the estimate is sharp (see \cite{Lehto-1977, Nehari-1976, Robertson-1969}). In 1996, Suita \cite{Suita-1996} studied the class $\mathcal{C}(\alpha)$, $0\le \alpha< 1$ and proved that the Schwarzian norm satisfies the following sharp inequality
$$||S_f||\le
\begin{cases}
2 & \text{ if }~ 0\le \alpha\le 1/2,\\
8\alpha(1-\alpha) & \text{ if }~ 1/2< \alpha<1.
\end{cases}
$$

A function $f\in\mathcal{A}$ is said to be strongly starlike (respectively, strongly convex) of order $\alpha$, $0<\alpha\le 1$ if $|\arg\{zf'(z)/f(z)\}|<\pi\alpha/2$ (respectively, $|\arg\{1+zf''(z)/f'(z)\}|<\pi\alpha/2$) for $z\in\mathbb{D}$. The classes of all strongly starlike and strongly convex functions of order $\alpha$ are denoted by $\mathcal{S}^*_{\alpha}$ and $\mathcal{K}_{\alpha}$, respectively. Chiang \cite{Chiang-1991} studied the Schwarzian norm for the class $\mathcal{S}^*_{\alpha}$ and proved $||S_f||\le 6\sin(\pi\alpha/2)$. In 2011, Kanas and Sugawa \cite{Kanas-Sugawa-2011} studied the Schwarzian norm for the class $\mathcal{K}_{\alpha}$ and proved the sharp inequality $||S_f||\leq 2\alpha$.\\

A function $f\in \mathcal{A}$ is said to be uniformly convex function if every circular arc (positively oriented) of the form $\{z\in\mathbb{D}: |z -\xi|=r\}$, $\xi\in\mathbb{D}$, $0<r<|\xi|+1$ is mapped by $f$ univalently onto a convex arc. The class $\mathcal{UCV}$ denotes of all uniformly convex functions. It is easy to prove that $\mathcal{UCV}\subset\mathcal{C}$. It is well known that (see \cite{Goodman-1991,Ma-Minda-1992,Ronning-1993}) a function $f\in \mathcal{A}$ is in $\mathcal{UCV}$ if and only if
$${\rm Re\,}\left(1+\frac{zf''(z)}{f'(z)}\right)> \left|\frac{zf''(z)}{f'(z)}\right|~~\text{for}~ z\in\mathbb{D}.$$
Also, in the same paper, Kanas and Sugawa proved the sharp estimate of the Schwarzian norm $||S_f||\le 8/\pi^2$ for functions in the class $\mathcal{UCV}$. In 2012, Bhowmik and Wirths \cite{Bhowmik-Wirths-2012} studied the class of concave functions $\mathcal{C}o(\alpha)$ for $1\le \alpha\le 2$ and obtained the sharp estimate $||S_f||\le 2(\alpha^2-1)$ for $f\in\mathcal{C}o(\alpha)$.  Recently, Ali and Pal \cite{Ali-Sanjit-2022a} studied the classes $\mathcal{G}(\beta)$ with $\beta>0$ and $\mathcal{F}(\alpha)$ with $-1/2\le \alpha\le 0$, consisting of functions in $\mathcal{A}$ that satisfy the relations ${\rm Re\,}\left(1+zf''(z)/f'(z)\right)<1+\beta/2$, and ${\rm Re\,}\left(1+zf''(z)/f'(z)\right)>\alpha$ for $z\in\mathbb{D}$, respectively, and proved the sharp estimates 
$$||S_f||\le 2\beta(\beta+2)\quad \text{for}~ f\in\mathcal{G}(\beta),$$
and $$||S_f||\le \frac{2(1-\alpha)}{1+\alpha}\quad \text{for}~ f\in\mathcal{F}(\alpha).$$
In 2023, Ali and Pal \cite{Ali-Sanjit-2023c} studied the class $\mathcal{S}_{\alpha}$ and proved that the sharp estimate of the Schwarzian norm satisfies the following inequalities
$$||S_f||\le
\begin{cases}
\dfrac{2\cos{\alpha}}{1-\sin{|\alpha|}}&~\text{for}~|\alpha|\le\dfrac{\pi}{6},\\[3mm]
8\cos{\alpha}\sin{|\alpha|} &~\text{for}~|\alpha|>\dfrac{\pi}{6}.
\end{cases}
$$
 A function $f\in\mathcal{A}$ is said to be  Janowski convex function if $1+zf''(z)/f'(z)\prec(1+Az)/(1+Bz)$, where $-1\le B<A\le 1$. The class of all Janowski convex functions is denoted by $\mathcal{C}(A,B)$ (see \cite{Janowski-1973a,Janowski-1973b}).  Also, Ali and Pal \cite{Ali-Sanjit-2022b} studied this class and obtained the sharp estimate of the Schwarzian norm for functions in $\mathcal{C}(A,B)$.\\
 
 Before we move to our main results, let us discuss some fundamental lemmas that we will utilize throughout the article to derive our primary results. Let $\mathcal{B}$ be the set of all analytic functions $\omega:\mathbb{D}\rightarrow\mathbb{D}$ and $\mathcal{B}_0$ be the subclass of $\mathcal{B}$ consists of functions $\omega$ such that $\omega(0)=0$. Functions in $\mathcal{B}_0$ are called Schwarz functions. According to Schwarz's lemma, if $\omega\in\mathcal{B}_0$, then $|\omega(z)|\le |z|$ and $|\omega'(0)|\le 1$. The equality occurs in any one of the inequalities if and only if $\omega(z)=e^{i\alpha}z$, $\alpha\in\mathbb{R}$. An extension of Schwarz lemma, known as Schwarz-Pick lemma that gives $|\omega'(z)|\le (1-|\omega(z)|^2)/(1-|z|^2)$, $z\in\mathbb{D}$ and $\omega\in\mathcal{B}$. In 1931, Dieudonn\'{e} \cite{Dieudonne-1931} obtained the exact region of the variability of $\omega'(z_0)$ for a given $z_0\in\mathbb{D}$ over the class $\mathcal{B}_0$.

\begin{lem}[Dieudonn\'{e}'s lemma]\cite{Dieudonne-1931, Duren-1983}\label{lem-001}
Let $\omega\in\mathcal{B}_0$ and $z_0\ne 0$ be a fixed point in $\mathbb{D}$. The region of values of $\omega'(z_0)$ is the closed disk
\begin{equation*}
\left|\omega'(z_0)-\frac{\omega(z_0)}{z_0}\right|\le \frac{|z_0|^2-|\omega(z_0)|^2}{|z_0|(1-|z_0|^2)}.
\end{equation*}
Moreover, the equality holds if and only if $\omega$ is a Blaschke product of degree $2$ fixing $0$.
\end{lem}

Dieudonn\'{e}'s lemma is an extension of both Schwarz's and Schwarz-Pick's lemma. Here, we remark that a Blaschke product of degree $n\in\mathbb{N}$ is of the form
$$B(z)=e^{i\theta}\prod_{j=1}^{n}\frac{z-z_j}{1-\bar{z_j}z},\quad z,z_j\in\mathbb{D}, ~\theta\in\mathbb{R}.$$
Dieudonn\'{e}'s lemma will become very important in proving our primary results and we also utilize this lemma to generate the extremal functions by using the Blaschke product.\\

In the present study, we mainly focus on the class $\mathcal{S}_{\alpha}(\beta)$ and aim to determine the estimate of the modulus of the Schwarzian derivative for functions in the class $\mathcal{S}_{\alpha}(\beta)$, where $-\pi/2<\alpha<\pi/2$ and $0\le \beta<1$. By using Dieudonn\'{e}'s lemma, we will prove this estimate is sharp for real values of $z$. This result will yield a sharp estimate of the Schwarzian norm for functions in this class. Also, we determine the sharp estimate of the pre-Schwarzian norm for functions in the class  $\mathcal{S}_{\alpha}(\beta)$, where $|\alpha|<\pi/2$ and $\beta\in [0,1)$.

\section{Main Results}

\begin{thm}\label{Thm-t-0005}
For $-\pi/2<\alpha<\pi/2$ and  $0\le \beta<1$, let $f$ be in the class $\mathcal{S}_{\alpha}(\beta)$. Then the Schwarzian derivative $S_f$ satisfies the following inequalities:

\begin{enumerate}
\item[(i)] If  $\sin^2{\alpha}+\beta^2\cos^2{\alpha}\le 1/4$, then
\begin{equation}\label{t-00015}
|S_f(z)|\le \dfrac{2(1-\beta)\cos{\alpha}\{1-(1-|z|^2)\sqrt{\sin^2{\alpha}+\beta^2\cos^2{\alpha}}\}}{(1-|z|^2)^2(1-\sqrt{\sin^2{\alpha}+\beta^2\cos^2{\alpha}})}\quad\text{for}~z\in\mathbb{D}.
\end{equation}

\item[(ii)] If $\sin^2{\alpha}+\beta^2\cos^2{\alpha}> 1/4$ and $z\in\mathbb{D}$, then
\begin{equation}\label{t-00020}
|S_f(z)|\le
\begin{cases}
\dfrac{2(1-\beta)\cos{\alpha}\{1-(1-|z|^2)\sqrt{\sin^2{\alpha}+\beta^2\cos^2{\alpha}}\}}{(1-|z|^2)^2(1-\sqrt{\sin^2{\alpha}+\beta^2\cos^2{\alpha}})}& ~\text{for}~|z|<\lambda,\\[6mm]
\dfrac{2(1-\beta)\cos{\alpha}\sqrt{\sin^2{\alpha}+\beta^2\cos^2{\alpha}}}{(1-|z|)^2}&~ \text{for}~|z|\ge\lambda,
\end{cases}
\end{equation}
where
\begin{equation}\label{t-00025}
\lambda=\dfrac{1-\sqrt{\sin^2{\alpha}+\beta^2\cos^2{\alpha}}}{\sqrt{\sin^2{\alpha}+\beta^2\cos^2{\alpha}}}.
\end{equation}
\end{enumerate}
\end{thm}

\begin{proof}
For $-\pi/2<\alpha<\pi/2$ and  $0\le \beta<1$, let $f\in\mathcal{S}_{\alpha}(\beta)$. Then from \eqref{t-00010}, we have
$$1+\frac{zf''(z)}{f'(z)}\prec \frac{1+Az}{1-z},$$
where
$$A=e^{-i\alpha}(e^{-i\alpha}-2\beta\cos{\alpha}).$$
Thus, there exists an analytic function $\omega:\mathbb{D}\rightarrow\mathbb{D}$ with $\omega(0)=0$ such that
$$1+\frac{zf''(z)}{f'(z)}= \frac{1+A\omega(z)}{1-\omega(z)}.$$
A simple computation gives
\begin{equation}\label{t-00030}
\frac{f''(z)}{f'(z)}=\frac{(A+1)\omega(z)}{z(1-\omega(z))},
\end{equation}
and therefore,
\begin{align*}
S_f(z)&=\left(\frac{f''(z)}{f'(z)}\right)^{'}-\frac{1}{2}\left(\frac{f''(z)}{f'(z)}\right)^2\\[3mm]
&=(A+1)\left\{\dfrac{\omega'(z)-\dfrac{\omega(z)}{z}}{z(1-\omega(z))^2}-\dfrac{(A-1)\omega^2(z)}{2z^2(1-\omega(z))^2}\right\}.\nonumber
\end{align*}
Moreover, the Schwarzian derivative $S_f(z)$ can be written in the following form
\begin{equation}\label{t-00035}
S_f(z)=(A+1)\left\{\dfrac{\omega'(z)-\dfrac{\omega(z)}{z}}{z(1-\omega(z))^2}-\dfrac{|A-1|e^{i\theta}\omega^2(z)}{2z^2(1-\omega(z))^2}\right\}, \quad \theta=Arg(A-1).
\end{equation}
Let us consider the transformation $\displaystyle e^{i\theta}\zeta(z)=\omega'(z)-\frac{\omega(z)}{z}$. By Dieudonn\'{e}'s Lemma \ref{lem-001}, the function $\zeta(z)$ varies over the closed disk
$$|\zeta(z)|\le \frac{|z|^2-|\omega(z)|^2}{|z|(1-|z|^2)}\quad\text{for fixed }|z|<1.$$
Using the transformation $\zeta(z)$ in \eqref{t-00035}, we obtain
\begin{equation}\label{t-00038}
S_f(z)=e^{i\theta}(A+1)\left\{\frac{\zeta(z)}{z(1-\omega(z))^2}-\frac{|A-1|\omega^2(z)}{2z^2(1-\omega(z))^2}\right\},
\end{equation}
and therefore,
\begin{align*}
|S_f(z)|&\le |A+1|\left\{\frac{|\zeta(z)|}{|z||1-\omega(z)|^2}+\frac{|A-1||\omega(z)|^2}{|z|^2|1-\omega(z)|^2}\right\}\\[3mm]
&\le |A+1|\left\{\frac{|z|^2-|\omega(z)|^2}{|z|^2(1-|z|^2)(1-|\omega(z)|)^2}+\frac{|A-1||\omega(z)|^2}{|z|^2(1-|\omega(z)|)^2}\right\}.
\end{align*}
For $0\le s:=|\omega(z)|\le |z|<1$, we obtain
\begin{align}\label{t-00040}
|S_f(z)|&\le |A+1|\left\{\frac{|z|^2-s^2}{|z|^2(1-|z|^2)(1-s)^2}+\frac{|A-1|s^2}{|z|^2(1-s)^2}\right\}\\[3mm]
&=|A+1|\left\{\frac{2|z|^2-s^2(2-|A-1|(1-|z|^2))}{2|z|^2(1-|z|^2)(1-s)^2}\right\}=|A+1|g(s),\nonumber
\end{align}
where
\begin{equation}\label{t-00045}
g(s)=\frac{2|z|^2-s^2(2-|A-1|(1-|z|^2))}{2|z|^2(1-|z|^2)(1-s)^2}, \quad 0\le s\le |z|<1.
\end{equation}
Now, we want to determine the maximum value of $g(s)$ in $[0,|z|]$. To do this, we have to first identify the critical values of $g(s)$ in $(0,|z|)$. Therefore,
\begin{equation*}
g'(s)=\frac{2|z|^2-s(2-|A-1|(1-|z|^2))}{|z|^2(1-|z|^2)(1-s)^3},
\end{equation*}
and so, $g'(s)=0$ gives
$$s=\frac{2|z|^2}{2-|A-1|(1-|z|^2)}=:s_0(|z|).$$
Now, we have to check for what values of $\alpha$ and $\beta$ the point $s_0(|z|)$ lies in $(0,|z|)$. We note that the inequality
\begin{equation}\label{t-00050}
s_0(|z|)=\frac{2|z|^2}{2-|A-1|(1-|z|^2)}<|z|,
\end{equation}
holds if and only if $h(|z|)>0$, where
\begin{equation}\label{t-00053}
h(t)=2-|A-1|(1+t)>0.
\end{equation}
If $A=1$ then it is clear that $h(|z|)>0$ and consequently, $s_0(|z|)$ lies in $(0,|z|)$. If $A\ne 1$, then $h(t)$ has unique zero on the positive real axis, say $\lambda$, it is given by 
$$\lambda:=\frac{2-|A-1|}{|A-1|}=\dfrac{1-\sqrt{\sin^2{\alpha}+\beta^2\cos^2{\alpha}}}{\sqrt{\sin^2{\alpha}+\beta^2\cos^2{\alpha}}}.$$
It is easy to see that $\lambda$ lies in $(0,1)$ if and only if $\sin^2{\alpha}+\beta^2\cos^2{\alpha}>1/4$. Now, we complete the proof by considering two different cases.\\

\noindent\textbf{Case-I:} Let $\sin^2{\alpha}+\beta^2\cos^2{\alpha}\le 1/4$. In this case, $h(|z|)>0$ in $(0,1)$ and so, $s_0(|z|)$ lies in $(0,|z|)$. Since, the numerator of $g'(s)$ is a linear function of $s$, and $g'(0)=2/(1-|z|)^2>0$, $g'(s_0(|z|))=0$, it follows that $g'(|z|)<0$. Therefore, the function $g$ is strictly increasing in $(0,s_0(|z|))$ and strictly decreasing in $(s_0(|z|),|z|)$ and therefore, $g$ attain its maximum at $s_0(|z|)$. Therefore, from \eqref{t-00040}, we get
\begin{align*}
|S_f(z)|&\le \dfrac{|A+1|(2-|A-1|(1-|z|^2))}{(1-|z|^2)^2(2-|A-1|)}\\[2mm]
&=\dfrac{2(1-\beta)\cos{\alpha}\{1-(1-|z|^2)\sqrt{\sin^2{\alpha}+\beta^2\cos^2{\alpha}}\}}{(1-|z|^2)^2(1-\sqrt{\sin^2{\alpha}+\beta^2\cos^2{\alpha}})}\quad\text{for}~z\in\mathbb{D}.
\end{align*}

\noindent\textbf{Case-II:} Let $\sin^2{\alpha}+\beta^2\cos^2{\alpha}>1/4$. In this case $h(|z|)>0$ for $|z|<\lambda$ and $h(|z|)\le 0$ for $|z|\ge\lambda$. Thus, for $|z|<\lambda$, following the same argument as in Case-I, the function $g$ is strictly increasing in $(0,s_0(|z|))$ and strictly decreasing in $(s_0(|z|),|z|)$ and therefore, $g$ attain its maximum at $s_0(|z|)$. Thus, the first inequality of \eqref{t-00020} follows from \eqref{t-00040} and \eqref{t-00045}. Since $h(|z|)\le 0$ for $\lambda\le |z|<1$, it follows that $|A-1|(1-|z|^2)\ge 2(1-|z|)$. Thus, for $|z|\ge \lambda$,  from \eqref{t-00045}, we obtain
 $$g(|z|)=\frac{|A-1|}{2(1-|z|)^2} = \frac{|A-1|(1-|z|^2)}{2(1-|z|)^2}g(0) \ge \frac{g(0)}{1-|z|}\ge g(0).$$
Thus, the second inequality of \eqref{t-00020} follows from \eqref{t-00040} and \eqref{t-00045}.\\

\end{proof}

Before proceed further, we will discuss the sharpness of the estimate $|S_f(z)|$ given in Theorem \ref{Thm-t-0005}. First, we will prove the sharpness of $|S_f(z)|$ given in \eqref{t-00015} and the first inequality of \eqref{t-00020} for real values of $z$. Let us consider the function $f_{z_0}$ with $-1<z_0<1$ that satisfy the following relation
\begin{equation}\label{t-00060}
1+\frac{zf_{z_0}''(z)}{f_{z_0}'(z)}= \frac{1+A\phi(z)}{1-\phi(z)},
\end{equation}
where 
$$A=e^{-i\alpha}(e^{-i\alpha}-2\beta\cos{\alpha})\quad\text{with}\quad \phi(z)=-\frac{z(z-b)}{1-bz}$$ 
and $b$ is a solution of the equation
\begin{align}\label{t-00065}
&-\frac{z_0(z_0-b)}{1-bz_0}=\frac{2z_0^2}{2-|A-1|(1-z_0^2)}\nonumber\\[3mm]
&i.e.,\quad b=\frac{z_0(4-|A-1|(1-z_0^2))}{2(1+z_0^2)-|A-1|(1-z_0^2)}.
\end{align}
It is well known that whenever $\alpha$, $\beta$ and $z_0$ satisfies the conditions of  $|S_f(z)|$ given in \eqref{t-00015} and the first inequality in \eqref{t-00020}, the point $s_0(|z_0|)$ lies in $[0,|z|)$, where $s_0(|z_0|)$ is given by \eqref{t-00050}. This lead us to conclude that $b\in(-1,1)$ and consequently, $\phi(z)$ is Blaschke product of degree $2$ fixing $0$. Thus, the function $f_{z_0}$ belongs to $\mathcal{S}_{\alpha}(\beta)$. Thus, from \eqref{t-00038}, the Schwarzian derivative of $f_{z_0}$ is given by
\begin{align*}
S_{f_{z_0}}(z)&=\left(\frac{f''_{z_0}(z)}{f'_{z_0}(z)}\right)^{'}-\frac{1}{2}\left(\frac{f''_{z_0}(z)}{f'_{z_0}(z)}\right)^2\\[2mm]
&=e^{i\theta}(A+1)\left\{\frac{\zeta(z)}{z(1-\phi(z))^2}-\frac{|A-1|\phi^2(z)}{2z^2(1-\phi(z))^2}\right\},
\end{align*}
where 
$$\zeta(z)=e^{-i\theta}\left(\phi'(z)-\frac{\phi(z)}{z}\right) \quad \text{and}\quad \theta=Arg(A-1).$$
Therefore, the Schwarzian derivative $S_{f_{z_0}}(z)$ at $z_0$ is given by
\begin{align*}
S_{f_{z_0}}(z_0)&=e^{i\theta}(A+1)\left\{\frac{\zeta(z_0)}{z_0(1-\phi(z_0))^2}-\frac{|A-1|\phi^2(z_0)}{2z^2(1-\phi(z_0))^2}\right\}\\[2mm]
&=e^{i\theta}(A+1)\left\{\frac{-2(1-b^2)-|A-1|(z_0-b)^2}{2(1-bz_0+z_0(z_0-b))^2}\right\}.
\end{align*}
Substituting the value of $b$ given in \eqref{t-00065} and calculating the Schwarzian derivative $f_{z_0}(z_0)$, we obtain
$$S_{f_{z_0}}(z_0)=\dfrac{e^{i\theta}(A+1)(2-|A-1|(1-z_0^2))}{(1-z_0^2)^2(2-|A-1|)}.$$
Therefore,
\begin{align}\label{t-00070}
|S_{f_{z_0}}(z_0)|&= \dfrac{|A+1|(2-|A-1|(1-|z_0|^2))}{(1-|z_0|^2)^2(2-|A-1|)}\nonumber\\[2mm]
&=\dfrac{2(1-\beta)\cos{\alpha}\{1-(1-|z_0|^2)\sqrt{\sin^2{\alpha}+\beta^2\cos^2{\alpha}}\}}{(1-|z_0|^2)^2(1-\sqrt{\sin^2{\alpha}+\beta^2\cos^2{\alpha}})}.
\end{align}
This shows that the bounds obtained in the inequality \eqref{t-00015} and the first inequality in \eqref{t-00020} are sharp for real $z$.\\

Now, we will discuss the sharpness of the second inequality of \eqref{t-00020}. To do this, let us consider the function $f_0(z)$ given by
\begin{equation}\label{t-00075}
1+\frac{zf_0''(z)}{f_0'(z)}= \frac{1+Az}{1-z},
\end{equation}
where 
$$A=e^{-i\alpha}(e^{-i\alpha}-2\beta\cos{\alpha}).$$ 
Using \eqref{t-00038} with $\omega(z)=z$, the Schwarzian derivative of $f_0(z)$ is given by
$$S_{f_0}(z)=-\frac{e^{i\theta}(A+1)|A-1|}{2(1-z)^2},\quad \theta=Arg(A-1).$$
Therefore, for any $\alpha$ and $\beta$ with $\sin^2{\alpha}+\beta^2\cos^2{\alpha}>1/4$, and $|z_0|\ge\lambda$ with $0< z_0< 1$, we have
$$|S_{f_0}(z_0)|=\frac{|A^2-1|}{2(1-|z_0|)^2}=\dfrac{2(1-\beta)\cos{\alpha}\sqrt{\sin^2{\alpha}+\beta^2\cos^2{\alpha}}}{(1-|z_0|)^2},$$
which shows that the bound obtained in the second inequality in \eqref{t-00020} is sharp for positive $z$.\\

The previous discussion proves that for certain real values of $z$, the estimate of the modulus of the Schwarzian derivative $|S_f(z)|$ given in Theorem \ref{Thm-t-0005} is sharp. This also allows us to comment on the estimate of the Schwarzian norm $||S_f||$ is sharp for functions in $\mathcal{S}_{\alpha}(\beta)$.

\begin{thm}\label{Thm-t-00010}
For $-\pi/2<\alpha<\pi/2$ and  $0\le \beta<1$, let $f$ be in the class $\mathcal{S}_{\alpha}(\beta)$. Then the Schwarzian norm $||S_f||$ satisfies the following inequality
$$||S_f||\le
\begin{cases}
\dfrac{2(1-\beta)\cos{\alpha}}{1-\sqrt{\sin^2{\alpha}+\beta^2\cos^2{\alpha}}}&~\text{for}~\sin^2{\alpha}+\beta^2\cos^2{\alpha}\le 1/4,\\[5mm]
8(1-\beta)\cos{\alpha}\sqrt{\sin^2{\alpha}+\beta^2\cos^2{\alpha}} &~\text{for}~\sin^2{\alpha}+\beta^2\cos^2{\alpha}> 1/4.
\end{cases}
$$
Moreover, the estimates are best possible.
\end{thm}

\begin{proof}
We prove the theorem by considering two different cases.\\

\noindent \textbf{Case-I: } Let $\sin^2{\alpha}+\beta^2\cos^2{\alpha}\le 1/4$. Then from \eqref{t-00015}, the Schwarzian derivative of $f$ satisfies the following inequality
$$|S_f(z)|\le \dfrac{2(1-\beta)\cos{\alpha}\{1-(1-|z|^2)\sqrt{\sin^2{\alpha}+\beta^2\cos^2{\alpha}}\}}{(1-|z|^2)^2(1-\sqrt{\sin^2{\alpha}+\beta^2\cos^2{\alpha}})}\quad\text{for}~z\in\mathbb{D},$$
Therefore,
\begin{align*}
||S_f||&=\sup\limits_{z\in\mathbb{D}}(1-|z|^2)^2|S_f(z)|\\
&\le \dfrac{2(1-\beta)\cos{\alpha}}{1-\sqrt{\sin^2{\alpha}+\beta^2\cos^2{\alpha}}}\sup\limits_{0\le |z|<1}\left\{1-(1-|z|^2)\sqrt{\sin^2{\alpha}+\beta^2\cos^2{\alpha}}\right\}\\
&=\dfrac{2(1-\beta)\cos{\alpha}}{1-\sqrt{\sin^2{\alpha}+\beta^2\cos^2{\alpha}}}.
\end{align*}
To show that the estimate is best possible, consider the function $f_{z_0}(z)\in\mathcal{S}_{\alpha}(\beta)$ with $-1<z_0<1$ given in \eqref{t-00060}. From \eqref{t-00070}, we obtain
\begin{align*}
(1-|z_0|^2)^2|S_{f_{z_0}}(z_0)|&=\dfrac{2(1-\beta)\cos{\alpha}\{1-(1-|z_0|^2)\sqrt{\sin^2{\alpha}+\beta^2\cos^2{\alpha}}\}}{1-\sqrt{\sin^2{\alpha}+\beta^2\cos^2{\alpha}}}\\
&\rightarrow \dfrac{2(1-\beta)\cos{\alpha}}{1-\sqrt{\sin^2{\alpha}+\beta^2\cos^2{\alpha}}}\quad\text{as }z_0\rightarrow 1^-.
\end{align*}
This shows that the estimate is best possible.\\

\noindent \textbf{Case-II: } Let $\sin^2{\alpha}+\beta^2\cos^2{\alpha}> 1/4$. Then from \eqref{t-00020}, the Schwarzian derivative of $f$ satisfies the following inequality
\begin{equation*}
|S_f(z)|\le
\begin{cases}
\dfrac{2(1-\beta)\cos{\alpha}\{1-(1-|z|^2)\sqrt{\sin^2{\alpha}+\beta^2\cos^2{\alpha}}\}}{(1-|z|^2)^2(1-\sqrt{\sin^2{\alpha}+\beta^2\cos^2{\alpha}})}& ~\text{for}~|z|<\lambda,\\[6mm]
\dfrac{2(1-\beta)\cos{\alpha}\sqrt{\sin^2{\alpha}+\beta^2\cos^2{\alpha}}}{(1-|z|)^2}&~ \text{for}~|z|\ge\lambda,
\end{cases}
\end{equation*}
where $\lambda$ is given by \eqref{t-00025}. Therefore,
\begin{equation}\label{t-00080}
||S_f||=\sup\limits_{z\in\mathbb{D}}(1-|z|^2)^2|S_f(z)|\le \max\{M_1,M_2\},
\end{equation}
where
\begin{align*}
M_1&=2(1-\beta)\cos{\alpha}\sup\limits_{0\le |z|<\lambda}\dfrac{1-(1-|z|^2)\sqrt{\sin^2{\alpha}+\beta^2\cos^2{\alpha}}}{1-\sqrt{\sin^2{\alpha}+\beta^2\cos^2{\alpha}}}\\[2mm]
&=\dfrac{2(1-\beta)\cos{\alpha}\{1-(1-\lambda^2)\sqrt{\sin^2{\alpha}+\beta^2\cos^2{\alpha}}\}}{1-\sqrt{\sin^2{\alpha}+\beta^2\cos^2{\alpha}}},
\end{align*}
and
\begin{align*}
M_2&=2(1-\beta)\cos{\alpha}\sqrt{\sin^2{\alpha}+\beta^2\cos^2{\alpha}}\sup\limits_{\lambda\le|z|<1}(1+|z|)^2\\
&=8(1-\beta)\cos{\alpha}\sqrt{\sin^2{\alpha}+\beta^2\cos^2{\alpha}}.
\end{align*}
Since $\lambda$ is a root of the equation $h(t)=0$, and therefore, $2-|A-1|(1-\lambda^2)=2\lambda$, where $A=e^{-i\alpha}(e^{-i\alpha}-2\beta\cos{\alpha})$. Using this fact, $M_1$ can be written as 
$$M_1=\frac{2\lambda|A+1|}{2-|A-1|}=\frac{2|A^2-1|}{|A-1|^2}=\frac{|A^2-1|}{2(\sin^2{\alpha}+\beta^2\cos^2{\alpha})}<2|A^2-1|=M_2.$$
Thus, the required result follows from \eqref{t-00080}.\\

Now, we show that the estimate is sharp. Let us consider the function $f_0\in\mathcal{S}_{\alpha}$ given in \eqref{t-00075}. Then the Schwarzian derivative of $f_0$ is given by
$$S_{f_0}(z)=-\frac{e^{i\theta}(A+1)|A-1|}{2(1-z)^2},\quad \theta=Arg(A-1).$$
Therefore,
$$||S_{f_0}||=\sup\limits_{z\in\mathbb{D}}(1-|z|^2)^2|S_{f_0}(z)|=\frac{|A^2-1|}{2}\sup\limits_{z\in\mathbb{D}}\frac{(1-|z|^2)^2}{|1-z|^2}.$$
On the positive real axis, we have
$$\sup\limits_{0\le t<1}\frac{(1-t^2)^2}{(1-t)^2}=4.$$
Thus,
$$||S_{f_0}||=2|A^2-1|=8(1-\beta)\cos{\alpha}\sqrt{\sin^2{\alpha}+\beta^2\cos^2{\alpha}}.$$
\end{proof}

For particular values $\alpha=\beta=0$, one can obtain the Schwarzian norm for functions in the class $\mathcal{C}$, which was first proved by Robertson \cite{Robertson-1969}.

\begin{cor}
If $f\in\mathcal{S}_0(0)=:\mathcal{C}$, then the Schwarzian norm satisfies the sharp inequality $||S_f||\le 2$.
\end{cor}

In particular, $\beta=0$, one can obtain the Schwarzian norm for functions in Robertson class $\mathcal{S}_{\alpha}$, which was obtained by Ali and Pal \cite{Ali-Sanjit-2023c}.

\begin{cor}
If $f\in\mathcal{S}_{\alpha}(0)=:\mathcal{S}_{\alpha}$, then the Schwarzian norm satisfies the sharp inequality 
$$||S_f||\le
\begin{cases}
\dfrac{2\cos{\alpha}}{1-\sin{|\alpha|}}&~\text{for}~|\alpha|\le \frac{\pi}{6},\\[5mm]
8\cos{\alpha}\sin{|\alpha|} &~\text{for}~|\alpha|>\frac{\pi}{6}.
\end{cases}
$$
\end{cor}

Now, we will focus on the pre-Schwarzian norm for functions in the class $\mathcal{S}_{\alpha}(\beta)$.

\begin{thm}\label{Thm-t-00015}
For $-\pi/2<\alpha<\pi/2$ and $0\le \beta<1$, let $f\in\mathcal{S}_{\alpha}(\beta)$. Then the pre-Schwarzian norm satisfies the following sharp inequality
$$||P_f||\le 4(1-\beta)\cos{\alpha}.$$
\end{thm}
\begin{proof}
For $-\pi/2<\alpha<\pi/2$ and $0\le \beta<1$, let $f\in\mathcal{S}_{\alpha}(\beta)$. Then from \eqref{t-00030}, the pre-Schwarzian derivative $P_f(z)$ is given by
$$P_f(z)=\frac{f''(z)}{f'(z)}=\frac{(A+1)\omega(z)}{z(1-\omega(z))},$$
where $\omega\in\mathcal{B}_0$ and $A=e^{-i\alpha}(e^{-i\alpha}-2\beta\cos{\alpha})$. Therefore,
\begin{align*}
||P_f||&=\sup\limits_{z\in\mathbb{D}}(1-|z|^2)|P_f(z)|\\
&\le |A+1|\sup\limits_{0\le |z|<1}(1+|z|)\\
&=2|A+1|=4(1-\beta)\cos{\alpha}.
\end{align*}
It is easy to verify that the equality occur for the function $f_0\in \mathcal{S}_{\alpha}$ defined by \eqref{t-00075}.

\end{proof}

For the particular value $\alpha=0$ and $\beta=0$, one can obtain the pre-Schwarzian norm for the functions in the class $\mathcal{C}$, which was first proved by Yamashita \cite{Yamashita-1999}.

\begin{cor}
If $f\in\mathcal{S}_0(0)=:\mathcal{C}$, then the pre-Schwarzian norm satisfies the sharp inequality $$||P_f||\le 4.$$
\end{cor}

For the particular value $\beta=0$, one can obtain the pre-Schwarzian norm for functions in Robertson class $\mathcal{S}_{\alpha}$, which was obtained by Ali and Pal \cite{Ali-Sanjit-2023c}.

\begin{cor}
If $f\in\mathcal{S}_{\alpha}(0)=:\mathcal{S}_{\alpha}$, then the pre-Schwarzian norm satisfies the sharp inequality $$||P_f||\le 4\cos{\alpha}.$$
\end{cor}

\vspace{4mm}
\textbf{Data availability:} Data sharing not applicable to this article as no data sets were generated or analysed during the current study.\\


\textbf{Acknowledgement:} The author thanks the University Grants Commission, India for the financial support through UGC Fellowship (Grant No. MAY2018-429303).


\begin{thebibliography}{99}



\bibitem{Ali-Sanjit-2022}
{\sc M. F. Ali and S. Pal}, Pre-Schwarzian norm estimates for the class of Janowski starlike functions, {\it Monatsh. Math.}  {\bf 201}(2) (2023), 311--327.



\bibitem{Ali-Sanjit-2022a}
{\sc M. F. Ali and S. Pal}, Schwarzian norm estimates for some classes of analytic functions, {\it Mediterr. J. Math.} {\bf 20}, 294 (2023).


\bibitem{Ali-Sanjit-2022b}
{\sc M. F. Ali and S. Pal}, The Schwarzian norm estimates for Janowski convex functions, arXiv:2212.06377.



\bibitem{Ali-Sanjit-2023c}
{\sc M. F. Ali and S. Pal}, Schwarzian norm estimate for functions in Robertson class, {\it Bull. Sci. Math.} {\bf 188} (2023), https://doi.org/10.1016/j.bulsci.2023.103335.



\bibitem{Becker-1972}
{\sc J. Becker}, L\"{o}wnersche Differentialgleichung und quasikonform fortsetzbare schlichte Funktionen, {\it J. Reine Angew. Math.} {\bf 255} (1972), 23--43.


\bibitem{Becker-Pommerenke-1984}
{\sc J. Becker} and {\sc  C. Pommerenke}, Schlichtheitskriterien und Jordangebiete, {\it J. Reine Angew. Math.} {\bf 354} (1984), 74--94.


\bibitem{Bhowmik-Wirths-2012}
{\sc B. Bhowmik and K. J. Wirths}, A sharp bound for the Schwarzian derivative of concave functions, {\it Colloq. Math.} {\bf 128}(2) (2012), 245--251.


\bibitem{Chiang-1991}
{\sc Y. M. Chiang}, {\it Schwarzian derivative and second order differential equations}, Ph.D. thesis, Univ. of London, 1991.



\bibitem{Chichra-1975}
{\sc P. N. Chichra}, Regular functions $f(z)$ for which $zf'(z)$ is $\alpha$-spiral, {\it Proc. Amer. Math. Soc.} {\bf 49} (1975), 151--160.



\bibitem{Dieudonne-1931}
{\sc J. A. Dieudonn{\'e}}, Recherches sur quelques probl{\`e}mes relatifs aux polyn{\^o}mes et aux fonctions born{\'e}es d'une variable complexe, {\it Ann. Sci. {\'E}c. Norm. Sup{\'e}r.} {\bf 48} (1931), 247--358.



\bibitem{Duren-1983}
{\sc P. L. Duren}, {\it Univalent functions} (Grundlehren der mathematischen Wissenschaften 259, New York, Berlin, Heidelberg, Tokyo), Springer-Verlag (1983).



\bibitem{Goodman-book-1983}
{\sc A. W. Goodman}, Univalent Functions, Vols. I and II. Mariner Publishing Co. Tampa, Florida (1983).


\bibitem{Goodman-1991}
{\sc A. W. Goodman}, On uniformly convex functions, {\it Ann. Polon. Math.} {\bf 56}(1) (1991), 87--92.



\bibitem{Janowski-1973a}
{\sc W. Janowski}, Some extremal problems for certain families of analytic functions I, {\it Ann. Polon. Math.} {\bf 28} (1973), 297--326.


\bibitem{Janowski-1973b}
{\sc W. Janowski}, Some extremal problems for certain families of analytic functions II, {\it Bull. Acad. Polon. Sci. Ser. Sci. Math. Astron. Phys.} {\bf 21} (1973) 17--25.


\bibitem{Kanas-Sugawa-2011}
{\sc S. Kanas, T. Sugawa}, Sharp norm estimate of Schwarzian derivative for a class of convex functions, {\it Ann. Polon. Math.} {\bf 101}(1) (2011), 75--86.



\bibitem{Kim-Sugawa-2002}
{\sc Y. C. Kim} and {T. Sugawa}, Growth and coefficient estimates for uniformly locally univalent functions on the unit disk, {\it Rocky Mountain J. Math.} {\bf 32} (2002), 179--200.


\bibitem{Kraus-1932}
{\sc W. Kraus}, {\"U}ber den Zusammenhang einiger Charakteristiken eines einfach zusammenh{\"a}ngenden Bereiches mit der Kreisabbildung, {\it Mitt. Math. Sem. Giessen} {\bf 21} (1932), 1--28.


\bibitem{Kuhnau-1971}
R. K\"{u}hnau, Verzerrungss\"{a}tze und Koeffizientenbedingungen vom Grunskyschen Typ f\"{u}r quasikonforme Abbildungen, {\it Math. Nachr.} {\bf 48} (1971), 77--105.


\bibitem{Lehto-1977}
{\sc O. Lehto}, Domain constants associated with Schwarzian derivative, {\it Comment. Math. Helv.} {\bf 52} (1977), 603--610.


\bibitem{Lehto-1987}
{\sc O. Lehto}, Univalent Functions and Teichm\"{u}ller Spaces, Springer-Verlag (1987).


\bibitem{Ma-Minda-1992}
{\sc W. Ma and D. Minda}, Uniformly convex functions, {\it Ann. Polon. Math.} {\bf 57} (1992), 165--175.


\bibitem{Libera-1967}
{\sc R. J. Libera}, Univalent $\alpha$-spiral functions, {\it Canad. J. Math.} {\bf 19} (1967), 449--456.


\bibitem{Nehari-1949}
{\sc Z. Nehari}, The Schwarzian derivative and schlicht functions, {\it Bull. Amer. Math. Soc.} {\bf 55}(6) (1949), 545--551.


\bibitem{Nehari-1976}
{\sc Z. Nehari}, A property of convex conformal maps, {\it J. Anal. Math.} {\bf 30} (1976), 390--393.


%



\bibitem{Pfaltzgraff-1975}
{\sc J.A. Pfaltzgraff}, Univalence of the integral of $f'(z)^{\lambda}$, {\it Bull. Lond. Math. Soc.} {\bf 7} (1975), 254--256.


\bibitem{Pinchuk-1971}
{\sc B. Pinchuk}, Functions of bounded boundary rotation, {\it Israel J. Math.} {\bf 10}(1) (1971), 6--16.



\bibitem{Robertson-1969}
{\sc M. S. Robertson}, Univalent functions $f(z)$ for which $zf'(z)$ is spirallike, {\it Michigan Math. J.} {\bf 16} (1969), 97--101.


\bibitem{Ronning-1993}
{\sc F. R\o nning}, Uniformly convex functions and a corresponding class of starlike functions, {\it Proc. Amer. Math. Soc.} {\bf 118} (1993), 189--196.


%
%



\bibitem{Spacek-1933}
{\sc L. \v{S}pa\v{c}ek}, Contribution \`{a} la th\'{e}orie des functions univalentes (in Czech), \v{C}asop P\v{e}st. Mat.-Fys, {\bf 62} (1933), 12--19.



\bibitem{Sugawa-1998}
{\sc T. Sugawa}, On the norm of pre-Schwarzian derivatives of strongly starlike functions, {\it Ann. Univ. Mariae Curie-Sk\l odowska Sect. A} {\bf 52}(2) (1998), 149--157.




\bibitem{Suita-1996}
{\sc N. Suita}, Schwarzian derivatives of convex functions,  {\it J. Hokkaido Univ. Ed. Sect. II A} {\bf 46} (1996), 113--117.







\bibitem{Yamashita-1976}
{\sc S. Yamashita}, Almost locally univalent functions, \textit{Monatsh. Math.} \textbf{81} (1976), 235--240.


\bibitem{Yamashita-1999}
{\sc S. Yamashita}, Norm estimates for function starlike or convex of order alpha, {\it Hokkaido Math. J.} {\bf 28}(1) (1999), 217--230.







\end{thebibliography}
\end{document}